\documentclass[12pt,oneside,english]{amsart}
\usepackage[T1]{fontenc}
\usepackage[latin9]{inputenc}
\usepackage{geometry}
\usepackage{color}
\usepackage{babel}
\usepackage{amstext}
\usepackage{amsthm}
\usepackage{amssymb}
\usepackage[unicode=true,pdfusetitle,
 bookmarks=true,bookmarksnumbered=false,bookmarksopen=false,
 breaklinks=false,pdfborder={0 0 0},pdfborderstyle={},backref=false,colorlinks=true]
 {hyperref}

\makeatletter
\numberwithin{equation}{section}
\numberwithin{figure}{section}
\theoremstyle{plain}
\newtheorem{thm}{\protect\theoremname}
\theoremstyle{plain}
\newtheorem{lem}[thm]{\protect\lemmaname}
\theoremstyle{plain}
\newtheorem{cor}[thm]{\protect\corollaryname}

\makeatother

\providecommand{\corollaryname}{Corollary}
\providecommand{\lemmaname}{Lemma}
\providecommand{\theoremname}{Theorem}

\begin{document}
\title{Some quantitative one-sided weighted estimates}
\begin{abstract}
In this paper we provide some quantitative one-sided estimates that
recover the dependences in the classical setting. Among them we provide
estimates for the one-sided maximal function in Lorentz spaces and
we show that the conjugation method for commutators works as well
in this setting.
\end{abstract}

\author{María Lorente}
\address{Departamento de Análisis Matemático, Estadística e Investigación Operativa
y Matemática Aplicada. Facultad de Ciencias. Universidad de Málaga
(Málaga, Spain)}
\email{m\_lorente@uma.es }
\author{Francisco J. Martín-Reyes}
\address{Departamento de Análisis Matemático, Estadística e Investigación Operativa
y Matemática Aplicada. Facultad de Ciencias. Universidad de Málaga
(Málaga, Spain)}
\email{martin\_reyes@uma.es}
\author{Israel P. Rivera-Ríos}
\address{Departamento de Análisis Matemático, Estadística e Investigación Operativa
y Matemática Aplicada. Facultad de Ciencias. Universidad de Málaga
(Málaga, Spain). Departamento de Matemática. Universidad Nacional
del Sur (Bahía Blanca, Argentina).}
\email{israelpriverarios@uma.es}
\maketitle

\section{Introduction and main results}

Given a weight $w$, namely a non-negative locally integrable function
we recall that $w\in A_{p}$ for $p>1$ if
\[
[w]_{A_{p}}=\sup_{Q}\frac{1}{|Q|}\int_{Q}w\left(\frac{1}{|Q|}\int_{Q}w^{-\frac{1}{p-1}}\right)^{p-1}<\infty
\]
and that $w\in A_{1}$ if 
\[
[w]_{A_{1}}=\left\Vert \frac{Mw}{w}\right\Vert _{L^{\infty}}<\infty
\]
where $M$ stands for the usual Hardy-Littlewood maximal function,
namely 
\[
Mw(x)=\sup_{x\in Q}\frac{1}{|Q|}\int_{Q}w
\]
where each $Q$ is a cube with its sides parallel to the axis. The
quantities $[w]_{A_{p}}$ are the so called $A_{p}$ constants. The
main feature of the $A_{p}$ classes is that they characterize the
weighted $L^{p}$ boundedness of $M$ and of non-degenerate Calderón-Zygmund
operators.

In the last decade quantitative weighted estimates in terms of the
$A_{p}$ constants have been a topic that has atracted the attention
of a number of authors. First results in that direction were obtained
in the early 90s in \cite{B}. Years later Astala, Iwaniec and Saksman
\cite{AIS} conjectured the linear dependence on the $A_{2}$ constant
for the Beurling transform. The fact that lead them to such a conjecture
was the fact that if such a result held, the solutions of the Beltrami
equation would have improved integrability properties. That conjecture
was settled in \cite{PV}. Soon a natural question arised, whether
the dependence on the $A_{2}$ constant was linear as well for general
Calderón-Zygmund operators. That question was solved in the positive
by Hytönen \cite{H} and lead to the development of sparse domination
theory, which has been a very fruitful tool in the field. 

Let us turn our attention now to the one-sided setting. We recall
that $w\in A_{p}^{+}$ for $p>1$ if 
\[
[w]_{A_{p}^{+}}=\sup_{a<b<c}\frac{1}{c-a}\int_{a}^{b}w\left(\frac{1}{c-a}\int_{b}^{c}w^{-\frac{1}{p-1}}\right)^{p-1}<\infty.
\]
Note that this classes of weights characterize the weighted boundedness
of the one-sided maximal function $M^{+}$. We recall that 
\[
M^{+}f(x)=\sup_{h>0}\frac{1}{h}\int_{x}^{x+h}|f(y)|dy,\qquad M^{-}f(x)=\sup_{h>0}\frac{1}{h}\int_{x-h}^{x}|f(y)|dy.
\]
We say that $w\in A_{1}^{+}$ if 
\[
[w]_{A_{1}^{+}}=\left\Vert \frac{M^{-}w}{w}\right\Vert _{L^{\infty}}<\infty.
\]
This class of weights characterizes the weighted weak type $(1,1)$
boundedness of $M^{+}$. 

At this point we would like to recall some quantitative weighted estimates.
In \cite{MRdT} de la Torre and Martín-Reyes showed that 
\begin{align}
\|M^{+}f\|_{L^{p}(w)} & \lesssim[w]_{A_{p}^{+}}^{\frac{1}{p-1}}\|f\|_{L^{p}(w)}.\nonumber \\
\|M^{+}f\|_{L^{p,\infty}(w)} & \lesssim[w]_{A_{p}^{+}}^{\frac{1}{p}}\|f\|_{L^{p}(w)}.\label{eq:WeakM+}
\end{align}
It is worth noting that this inequality matches the dependence in
the classical setting. One may think that the same should happen with
the remainder of operators and, up until now for all the operators
for which a quantitative counterpart has been settled, that has been
the case. Let us revisit some further results. 

In \cite{RV1} Riveros and Vidal showed that for one-sided fractional
integrals, if as usual, $0<\alpha<1$, $1<p<\frac{1}{\alpha}$ , $\frac{1}{q}=\frac{1}{p}-\alpha$
and $w\in A_{p,q}^{+}$ then 
\begin{align*}
\|I_{\alpha}^{+}f\|_{L^{q,\infty}(w^{q})} & \lesssim[w]_{A_{p,q}^{+}}^{1-\alpha}\|f\|_{L^{p}(w^{p})}\\
\|I_{\alpha}^{+}f\|_{L^{q}(w^{q})} & \lesssim[w]_{A_{p,q}^{+}}^{(1-\alpha)\max\left\{ 1,\frac{p'}{q}\right\} }\|f\|_{L^{p}(w^{p}).}
\end{align*}
See Section \ref{sec:Preliminaries} for the precise definition of
$A_{p,q}^{+}$ and of $I_{\alpha}^{+}$. 

Besides those estimates, in \cite{RV} they also provide sharp Coifman-Fefferman
estimates, and sharp $A_{1}^{+}$ estimates and they also show that
for every one-sided Calderón-Zygmund operator $T^{+}$ (see Section
\ref{sec:Preliminaries} for the precise definition of $T^{+}$)
\[
\|T^{+}f\|_{L^{1,\infty}(w)}\lesssim[w]_{A_{1}^{+}}\log\left(e+[w]_{A_{1}^{+}}\right)\|f\|_{L^{1}(w)}.
\]
As we already mentioned above, the dependences in all the aforementioned
estimates match their classical setting counterparts.

In contrast to what happens with the scalar setting, the $A_{2}$
conjecture for one-sided Calderón-Zygmund operators remains an open
problem. In that direction, the most recent advance is due to Chen,
Han and Lacey \cite{CHL} who showed for the martingale transform
$G$ that 
\begin{equation}
\|Gf\|_{L^{2}(w)}\lesssim[w]_{A_{2}^{+}}\|f\|_{L^{2}(w)}.\label{eq:A2Martingale}
\end{equation}

Now we present our contribution. Quite recently in \cite{ADNOP} sharp
weighted estimates in terms of the $A_{p}$ classes for the maximal
function on weighted Lorentz spaces were settled. Here we provide
the following one sided counterpart. 
\begin{thm}
\label{thm:LorentzAp+}Let $1<p<\infty$ and $q\in(0,\infty].$ Then,
if $w\in A_{p}^{+}$ we have that
\[
\|M^{+}\|_{L^{p,q}(w)\to L^{p,q}(w)}\lesssim\begin{cases}
(1+A)[\sigma]_{A_{\infty}^{-}}^{\frac{1}{p}}[w]_{A_{p}^{+}}^{\frac{1}{p}} & \text{if \ensuremath{p\leq q\leq\infty};}\\
(1+A)[\sigma]_{A_{\infty}^{-}}^{\frac{1}{q}}[w]_{A_{p}^{+}}^{\frac{1}{p}} & \text{if \ensuremath{0<q\leq p}}
\end{cases}
\]
where $\sigma=w^{1-p'}$. 
\end{thm}

We remit the reader to Section \ref{sec:Preliminaries} for the definition
of $A_{\infty}^{-}$.

In Subsection \ref{subsec:ResultJointAp+Ainfty-} we provide as well
a result assuming that we have a pair of weights $(v,w)$ that satisfy
the joint $A_{p}^{+}$ condition (see the definition in Section \ref{sec:Preliminaries})
and such that $w^{1-p'}\in A_{\infty}^{-}$. 

Our second contribution in this work is related to commutators. Coifman,
Rochberg and Weiss showed in their seminal work \cite{CRW} that $b\in BMO$
implies that the commutator 
\[
[b,T]f=bT(f)-T(bf)
\]
where $T$ is a singular integral operator is bounded on $L^{p}$.
They provided two proofs of that fact. The first one relied upon a
rather involved good-$\lambda$ type argument. The second one hinged
upon a very general approach which is nowadays known as the conjugation
method. Such a method has proven to be a powerful tool to provide
sharp weighted estimates for commutators. We remit the reader to \cite{BMMST}
and the references therein for a thorough treatise on this method. 

Our next Theorem is a one sided version of the conjugation method.
\begin{thm}
\label{thm:Conjugation}Let $1<p_{0},p,q<\infty$ and let $T$ be
an operator such that for every $w^{q}\in A_{p_{0}}^{+}$
\[
\|Tf\|_{L^{q}(w^{q})}\lesssim\varphi([w^{q}]_{A_{p_{0}}^{+}})\|f\|_{L^{p}(w^{p})}.
\]
Then there exists $\kappa>0$ such that
\[
\|T_{b}^{k}f\|_{L^{q}(w^{q})}\lesssim\varphi\left(\kappa[w^{q}]_{A_{p_{0}}^{+}}\right)[w^{q}]_{A_{p_{0}}^{+}}^{k\max\left\{ 1,\frac{1}{p_{0}-1}\right\} }\|b\|_{BMO}^{k}\|f\|_{L^{p}(w^{p})}
\]
where $T_{b}^{1}f=[b,T]f$ and $T_{b}^{k}f=[b,T_{b}^{k-1}]f$ for
$k>1$.
\end{thm}

We shall derive some corollaries from this result in Section \ref{sec:CorConjugation}.

The remainder of this paper is organized as follows. In Section \ref{sec:Preliminaries}
we gather some definitions and results that will be useful for our
purposes. We devote Section \ref{sec:ProofsMainResults} to settle
the main results. In the last section we provide some corollaries
and conjectures related to Theorem \ref{thm:Conjugation}.

\section{Preliminaries \label{sec:Preliminaries}}

\subsection{One-sided singular integral operators}

We recall that a function $K\in L_{\text{loc}}^{1}(\mathbb{R}\setminus\{0\})$
is a Calderón-Zygmund kernel if the following properties hold.
\begin{enumerate}
\item There exists $B_{1}>0$ such that 
\[
\left|\int_{\varepsilon<|x|<N}K(x)dx\right|\leq B_{1}
\]
for all $0<\varepsilon<N$. Also, $\lim_{\varepsilon\rightarrow0^{+}}\int_{\varepsilon<|x|<N}K(x)dx$
exists.
\item There exists $B_{2}>0$ such that
\[
|K(x)|\leq\frac{B_{2}}{|x|}
\]
for every $x\not=0$. 
\item There exists $B_{3}>0$ such that
\begin{equation}
|K(x-y)-K(x)|\leq B_{3}\frac{|y|}{|x|^{2}}\label{eq:HLipschitz}
\end{equation}
for every $x$ and $y$ with $|x|>2|y|>0$.
\end{enumerate}
We say that $T^{+}$ is a one-sided Calderón-Zygmund singular integral
operator if 
\begin{equation}
T^{+}f(x)=\lim_{\varepsilon\rightarrow0^{+}}\int_{x+\varepsilon}^{\infty}K(x-y)f(y)dy\label{eq:defT+}
\end{equation}
where $K$ is a Calderón-Zygmund kernel with support in $\mathbb{R}^{-}$. 

Even though these operators are still Calderón-Zygmund operators,
the fact that the kernel is supported just in $\mathbb{R}^{-}$ allows
to show that $A_{p}^{+}$ which is larger than $A_{p}$ is sufficient
for the boundedness of $T^{+}$. That result was obtained in \cite{AFMR}.

We may replace (\ref{eq:HLipschitz}) by some other smoothness conditions.
We say that $K$ satisfies an $L^{r}$-Hörmander condition if there
exist numbers $c_{r},C_{r}>0$ such that for any $y\in\mathbb{R}$
and $R>c_{r}|y|$, 
\begin{equation}
\sum_{m=1}^{\infty}2^{m}R\left(\frac{1}{2^{m}R}\int_{2^{m}R<|x|\leq2^{m+1}R}|K(x-y)-K(x)|^{r}dx\right)^{\frac{1}{r}}\leq C_{r}\label{eq:LrHormander}
\end{equation}
if $1\leq r<\infty$ and
\[
\sum_{m=1}^{\infty}2^{m}R\sup_{2^{m}R<|x|\leq2^{m+1}R}|K(x-y)-K(x)|\leq C_{\infty}
\]
if $r=\infty$. 

Given $1\leq r\leq\infty$, We say that $T^{+}$ is a one-sided $L^{r'}$-Hörmander
operator if $T$ admits an expression like (\ref{eq:defT+}) in terms
of a kernel that satisfies the Calderón-Zygmund conditions but with
(\ref{eq:HLipschitz}) replaced by the $L^{r'}$-Hörmander condition.
This class of one sided operators and slightly more general ones was
studied in \cite{LRdT}.

\subsection{One-sided weights}

\subsubsection{Some definitions}

Analogously as $A_{p}^{+}$ was defined, we may define $A_{p}^{-}$
just ``reversing'' the real line. We say that $w\in A_{p}^{-}$ if
\[
[w]_{A_{p}^{-}}=\sup_{a<b<c}\frac{1}{c-a}\int_{b}^{c}w\left(\frac{1}{c-a}\int_{a}^{b}w^{-\frac{1}{p-1}}\right)^{p-1}<\infty
\]
and that $w\in A_{1}^{-}$ if 
\[
[w]_{A_{1}^{-}}=\left\Vert \frac{M^{+}w}{w}\right\Vert _{L^{\infty}}<\infty.
\]
Note that the main feature of those classes is that they characterize
the weighted strong and weak type boundedness of $M^{-}$.

We recall that the one sided $A_{p,q}$ classes are defined by
\begin{align*}
[w]_{A_{p,q}^{+}} & =\sup_{a<b<c}\frac{1}{c-a}\int_{a}^{b}w^{q}\left(\frac{1}{c-a}\int_{b}^{c}w^{-p'}\right)^{\frac{q}{p'}}<\infty,\\{}
[w]_{A_{p,q}^{-}} & =\sup_{a<b<c}\frac{1}{c-a}\int_{b}^{c}w^{q}\left(\frac{1}{c-a}\int_{a}^{b}w^{-p'}\right)^{\frac{q}{p'}}<\infty.
\end{align*}
Note that these classes of weights characterize the weighted boundedness
of the one-sided fractional integrals $I_{\alpha}^{+}$ and $I_{\alpha}^{-}$
and of the maximal functions $M_{\alpha}^{+}$ and $M_{\alpha}^{-}$
respectively as it was shown in \cite{MRdTFrac}.

A fundamental property of those classes of weight is that if $r=1+\frac{q}{p'}$
then $[w]_{A_{p,q}^{+}}=[w^{q}]_{A_{r}^{+}}$ and also $[w]_{A_{p,q}^{+}}^{p'/q}=[w^{-p'}]_{A_{r'}^{-}}.$

\subsubsection{Reverse Hölder inequality}

Since $A_{p}$ classes are increasing it is natural to define $A_{\infty}=\bigcup_{p\geq1}A_{p}$.
It was shown in \cite{F,W} that 
\[
[w]_{A_{\infty}}=\sup_{x\in Q}\frac{1}{w(Q)}\int_{Q}M(\chi_{Q}w)<\infty
\]
characterizes the $A_{\infty}$ class. In \cite{HP} a suitable sharp
reverse Hölder inequality, namely
\[
\left(\frac{1}{|Q|}\int_{Q}w^{1+\delta}\right)^{\frac{1}{1+\delta}}\leq2\frac{1}{|Q|}\int_{Q}w
\]
provided $0\leq\delta\leq\frac{1}{\tau_{n}[w]_{A_{\infty}}}$, where
$\tau_{n}$ is a constant depending on the dimension of the space
$\mathbb{R}^{n}$, was settled and used to derive a number of mixed
constant quantitative estimates. It was also shown in \cite{HP} that
this $A_{\infty}$ constant is the smallest among the known ones characterizing
that class. 

In the one sided setting the one sided $A_{\infty}$ constants were
provided in \cite{MRdT}. We have that $w\in A_{\infty}^{-}$ if 
\[
[w]_{A_{\infty}^{-}}=\sup_{I}\frac{1}{w(I)}\int_{I}M^{+}(\chi_{I}w)<\infty
\]
and analogously, $w\in A_{\infty}^{+}$ if 
\[
[w]_{A_{\infty}^{+}}=\sup_{I}\frac{1}{w(I)}\int_{I}M^{-}(\chi_{I}w)<\infty.
\]
Note that as in the classical case for every $p$
\begin{equation}
[w]_{A_{\infty}^{+}}\leq c_{p}[w]_{A_{p}^{+}},\qquad[w]_{A_{\infty}^{-}}\leq c_{p}[w]_{A_{p}^{-}}.\label{eq:AinftyAp}
\end{equation}
Furthermore, a reverse Hölder type inequality was obtained as well
in \cite[Theorem 3.4]{MRdT}. We recall it in the following Lemma.
\begin{lem}
\label{lem:OneSidedRHI}There exists $\tau>0$ such that if $w\in A_{\infty}^{-}$
and $0<\varepsilon\leq\frac{1}{\tau[w]_{A_{\infty}^{-}}}$, then for
every $a<b<c$, 

\begin{equation}
|(a,b)|^{\varepsilon}\left(\int_{b}^{c}w^{1+\varepsilon}\right)\leq2\left(\int_{a}^{c}w\right)^{1+\varepsilon}\label{eq:RHAp-}
\end{equation}
and if $w\in A_{\infty}^{+}$ and $0<\varepsilon\leq\frac{1}{\tau[w]_{A_{\infty}^{+}}}$,
then for every $a<b<c$, 
\begin{equation}
|(b,c)|^{\varepsilon}\left(\int_{a}^{b}w^{1+\varepsilon}\right)\leq2\left(\int_{a}^{c}w\right)^{1+\varepsilon}.\label{eq:RHAp+}
\end{equation}
\end{lem}

It will be useful for our purposes to have the following $A_{p}$
version of the reverse Hölder inequality type above. 
\begin{lem}
\label{lem:OneSidedRHIAp}There exists $\tau_{p}>0$ such that if
$w\in A_{p}^{-}$ and $0<\varepsilon\leq\frac{1}{\tau_{p}[w]_{A_{p}^{-}}}$,
then for every $a<b<c$, 

\begin{equation}
|(a,b)|^{\varepsilon}\left(\int_{b}^{c}w^{1+\varepsilon}\right)\leq2\left(\int_{a}^{c}w\right)^{1+\varepsilon}\label{eq:RHAp-Ap}
\end{equation}
and if $w\in A_{p}^{+}$ and $0<\varepsilon\leq\frac{1}{\tau_{p}[w]_{A_{p}^{+}}}$,
then for every $a<b<c$, 
\begin{equation}
|(b,c)|^{\varepsilon}\left(\int_{a}^{b}w^{1+\varepsilon}\right)\leq2\left(\int_{a}^{c}w\right)^{1+\varepsilon}.\label{eq:RHAp+Ap}
\end{equation}
\end{lem}

Note that this Lemma is an straightforward consequence of the preceding
Lemma and of (\ref{eq:AinftyAp}).

\subsubsection{A quantitative extrapolation result}

We end this section recalling a quantitative extrapolation result
that was settled in \cite{RV1}.
\begin{thm}[{\cite[Theorem 4.1]{RV1}}]
\label{thm:SharpExt}Let $T$ be a sublinear operator defined on
$\mathcal{C}_{c}^{\infty}(\mathbb{R})$. If the inequality
\[
\|Tf\|_{L^{q_{0}}(w^{q_{0}})}\lesssim[w]_{A_{p_{0},q_{0}}^{+}}^{\gamma}\|f\|_{L^{p_{0}}(w^{p_{0}})}
\]
 holds for some pair $(p_{0},q_{0})$ with $1<p_{0}\leq q_{0}<\infty$
and for all weights $w$ in the class $A_{p_{0},q_{0}}^{+}$, then
for any pair $(p,q)$ with $1<p\leq q<\infty$, satisfying $\frac{1}{p}-\frac{1}{q}=\frac{1}{p_{0}}-\frac{1}{q_{0}}$
and for any weight $w\in A_{p,q}^{+}$ the inequality
\[
\|Tf\|_{L^{q}(w^{q})}\lesssim[w]_{A_{p,q}^{+}}^{\gamma\max\left\{ 1,\frac{q_{0}}{p'_{0}}\frac{p'}{q}\right\} }\|f\|_{L^{p}(w^{p})}
\]
holds, provided the left hand-side is finite. 
\end{thm}

\section{Proofs of the main results\label{sec:ProofsMainResults} }

\subsection{A key lemma}

In this section we present a result that essentially says that if
the one-sided $A_{p}$ condition holds with a ``gap'' then it actually
holds.
\begin{lem}
\label{lem:Gap}Let $p>1$ and let $t\geq2$ be an integer. Let $(v,w)$
be a pair of weights. If for every interval $I=(a,b)$ we have that
\begin{equation}
\frac{1}{\frac{l_{I}}{t}}\int_{a}^{a+\frac{l_{I}}{t}}v\left(\frac{1}{\frac{l_{I}}{t}}\int_{b-\frac{l_{I}}{t}}^{b}\sigma\right)^{p-1}\leq K,\label{eq:gapCondition}
\end{equation}
where $l_{I}=b-a$ and $\sigma=w^{1-p'}$, then
\[
[v,w]_{A_{p}}\leq K.
\]
\end{lem}

\begin{proof}
Given an interval $(a,c)$. Let $x\in(a,c)$. We call $x_{0}=a$,
$x_{k+1}-x_{k}=\frac{1}{t}(x-x_{k})$. Then we have that 
\[
(a,x)=\bigcup_{k=0}^{\infty}(x_{k},x_{k+1}].
\]
Now we fix $k$ and let $y<x$ such that $x-y=x_{k+1}-x_{k}$. Hence,
by (\ref{eq:gapCondition}),
\begin{align*}
v(x_{k},x_{k+1}) & \leq K(x_{k+1}-x_{k})\left(\frac{x-y}{\sigma(y,x)}\right)^{p-1}\\
 & \leq K(x_{k+1}-x_{k})\left(M_{\sigma}\left(\frac{1}{\sigma}\chi_{(a,x)}\right)(x)\right)^{p-1}.
\end{align*}
Summing in $k$, 
\[
v(a,x)\leq K(x-a)\left(M_{\sigma}\left(\frac{1}{\sigma}\chi_{(a,x)}\right)(x)\right)^{p-1}
\]
and consequently
\[
\frac{v(a,x)}{x-a}\leq K\left(M_{\sigma}\left(\frac{1}{\sigma}\chi_{(a,x)}\right)(x)\right)^{p-1}.
\]
Now let $a<b<c$. For every $x\in(b,c)$ we have that 
\begin{align*}
\frac{v(a,b)}{c-a}\leq\frac{v(a,x)}{x-a} & \leq K\left(M_{\sigma}\left(\frac{1}{\sigma}\chi_{(a,x)}\right)(x)\right)^{p-1}\\
 & \leq K\left(M_{\sigma}\left(\frac{1}{\sigma}\chi_{(a,c)}\right)(x)\right)^{p-1}.
\end{align*}
Note, that since 
\[
(b,c)\subset\left\{ x\,:\,M_{\sigma}\left(\frac{1}{\sigma}\chi_{(a,c)}\right)(x)\geq\left(\frac{v(a,b)}{K(c-a)}\right)^{\frac{1}{p-1}}\right\} ,
\]
by the weak type $(1,1)$ of $M_{\sigma}$ we have that 
\begin{align*}
\sigma(b,c) & \leq\sigma\left\{ x\,:\,M_{\sigma}\left(\frac{1}{\sigma}\chi_{(a,c)}\right)(x)\geq\left(\frac{v(a,b)}{K(c-a)}\right)^{\frac{1}{p-1}}\right\} \\
 & \leq\left(\frac{K(c-a)}{v(a,b)}\right)^{\frac{1}{p-1}}\int_{\mathbb{R}}\frac{1}{\sigma}\chi_{(a,c)}\sigma=\left(\frac{K(c-a)}{v(a,b)}\right)^{\frac{1}{p-1}}(c-a)
\end{align*}
from what readily follows that
\[
\frac{v(a,b)}{c-a}\left(\frac{\sigma(b,c)}{c-a}\right)^{p-1}\leq K
\]
and hence $[v,w]_{A_{p}}\leq K$ and we are done.
\end{proof}

\subsection{Proof of weighted Lorentz estimates for the maximal function}

In \cite{ADNOP} the authors provided a rather general result that
included the maximal function as a particular case, which is the following.
\begin{thm}
\cite[Theorem 2.1]{ADNOP}\label{thm:ADNOP} Let $(\Omega,\mu)$ be
a $\sigma$-finite measure space. Let $p\in(1,\infty)$, $q\in(0,\infty]$,
$r\in(1,p)$, $A>0$, and let $v,w$ be weights in $\Omega$. Suppose
$T$ is a positive sublinear operator in $L^{0}(\Omega)$ satisfying
the following properties: 
\begin{enumerate}
\item \label{enum:op1} $T$ is defined on the constant function $1$ with
$|T1|\leq A1$; 
\item \label{enum:op2} $T$ is bounded from $L^{p/r}(w)$ to $L^{p/r,\infty}(v)$. 
\end{enumerate}
Then 
\begin{align*}
\|T\|_{L^{p,q}(w)\to L^{p,q}(v)}\lesssim\begin{cases}
(1+A)(r')^{\frac{1}{p}}\|T\|_{L^{p/r}(w)\to L^{p/r,\infty}(v)}^{\frac{1}{r}} & \text{if \ensuremath{p\leq q\leq\infty};}\\
(1+A)\big(\frac{4r'}{q}\big)^{\frac{1}{q}}\|T\|_{L^{p/r}(w)\to L^{p/r,\infty}(v)}^{\frac{1}{r}} & \text{if \ensuremath{0<q\leq p}.}
\end{cases}
\end{align*}

\end{thm}

As a consequence of this result, combined with the weak type $(p,p)$
for the maximal function
\[
\|Mf\|_{L^{p,\infty}(w)}\lesssim[w]_{A_{p}}^{\frac{1}{p}}\|f\|_{L^{p}(w)}
\]
and the reverse Hölder inequality they derive sharp estimates for
$\|M\|_{L^{p,q}(w)\to L^{p,q}(w)}.$ 

\subsubsection{Proof of Theorem \ref{thm:LorentzAp+}}

For one sided weights the following result was obtained in \cite{MRdT}.
\begin{lem}[{\cite[Theorem 1.8]{MRdT}}]
\label{lem:RH}Let $1<p<+\infty$. If $w\in A_{p}^{+}$ and $\sigma=w^{1-p'}$,
$0<\delta\leq\frac{1}{2\tau[\sigma]_{A_{\infty}^{-}}}$ and $s=\frac{p+\delta}{(1+\delta)}$
then $w\in A_{s}^{+}$ and 
\[
[w]_{A_{s}^{+}}\leq2^{p}[w]_{A_{p}^{+}}.
\]
\end{lem}

Armed with that result we are in the position to give our proof of
Theorem \ref{thm:LorentzAp+}.
\begin{proof}[Proof of Theorem \ref{thm:LorentzAp+}]
Observe that for some $r>0$
\[
s=\frac{p}{r}\iff r=\frac{p}{s}.
\]
Then
\begin{align*}
r' & =\frac{p}{p-s}=\frac{p}{p-\frac{p+\delta}{(1+\delta)}}=\frac{p(1+\delta)}{p+p\delta-p-\delta}\\
 & =\frac{p(1+\delta)}{p\delta-\delta}=\frac{1+\delta}{\delta}p'=\left(1+\frac{1}{\delta}\right)p'.
\end{align*}
In particular if we choose $\delta=\frac{1}{2\tau[\sigma]_{A_{\infty}^{-}}}$
we have that 
\[
r'=\left(1+\frac{1}{\delta}\right)p'=\left(1+2\tau[\sigma]_{A_{\infty}^{-}}\right)p'\lesssim p'[\sigma]_{A_{\infty}^{-}}.
\]
Summarizing, we have that for $r=\frac{p}{s}$, 
\[
r'\lesssim p'[\sigma]_{A_{\infty}^{-}}
\]
and
\[
[w]_{A_{p/r}^{+}}\leq2^{p}[w]_{A_{p}^{+}}.
\]
Hence, plugging all these estimates in Theorem \ref{thm:ADNOP} and
taking into account (\ref{eq:WeakM+}), we arrive to the desired conclusion.
\end{proof}

\subsubsection{A result in terms of the joint $A_{p}^{+}$ condition and $A_{\infty}^{-}$\label{subsec:ResultJointAp+Ainfty-}}

It is possible to provide a result assuming the joint $A_{p}^{+}$
condition on the pair of weights $(v,w)$ and that $\sigma=w^{1-p'}\in A_{\infty}^{-}$.
The precise statement is the following.
\begin{thm}
\label{thm:LorentzAp+Joint}Let $1<p<\infty$ and $q\in(0,\infty].$
Then, if $(v,w)\in A_{p}^{+}$ and $\sigma=w^{1-p'}\in A_{\infty}^{-}$
we have that
\[
\|M^{+}\|_{L^{p,q}(w)\to L^{p,q}(w)}\lesssim\begin{cases}
(1+A)[\sigma]_{A_{\infty}^{-}}^{\frac{1}{p}}[v,w]_{A_{p}^{+}}^{\frac{1}{p}} & \text{if \ensuremath{p\leq q\leq\infty};}\\
(1+A)[\sigma]_{A_{\infty}^{-}}^{\frac{1}{q}}[v,w]_{A_{p}^{+}}^{\frac{1}{p}} & \text{if \ensuremath{0<q\leq p}}.
\end{cases}
\]
 
\end{thm}

In order to settle that result we will rely upon the following Lemma.
\begin{lem}
\label{lem:RHJoint}Let $p>1$ then, if $(v,w)$ satisfies the $A_{p}^{+}$
condition for pairs of weights, namely 
\[
[v,w]_{A_{p}^{+}}=\sup_{a<b<c}\frac{1}{c-a}\int_{a}^{b}v\left(\frac{1}{c-a}\int_{b}^{c}w^{-\frac{1}{p-1}}\right)^{p-1}<\infty
\]
and $\sigma=w^{1-p'}\in A_{\infty}^{-}$ we have that 
\[
[v,w]_{A_{p/r}^{+}}\leq6^{p-1}[v,w]_{A_{p}^{+}}
\]
for $r>1$ such that $\frac{p-1}{\frac{p}{r}-1}=r_{\sigma}\leq1+\frac{1}{\left[w^{-\frac{1}{p-1}}\right]_{A_{\infty}^{-}}2\tau}$. 
\end{lem}

Before settling the Lemma, note that if 
\[
\frac{p-1}{\frac{p}{r}-1}=1+\delta
\]
then we can show that 
\[
r'=p'\left(1+\frac{1}{\delta}\right)
\]
hence, choosing $\delta=\frac{1}{\left[w^{-\frac{1}{p-1}}\right]_{A_{\infty}^{-}}2\tau}$
we have that 
\[
r'=p'\left(1+\left[w^{-\frac{1}{p-1}}\right]_{A_{\infty}^{-}}2\tau\right)\lesssim p'\left[w^{-\frac{1}{p-1}}\right]_{A_{\infty}^{-}}.
\]

\begin{proof}
Observe that if given 
\[
a<b<c<d
\]
where
\[
b=a+\frac{d-a}{3}\qquad c=a+\frac{2}{3}(d-a)
\]
we have that 
\begin{equation}
\frac{1}{d-a}\int_{a}^{b}v\left(\frac{1}{d-a}\int_{c}^{d}w^{-\frac{1}{\frac{p}{r}-1}}\right)^{\frac{p}{r}-1}\leq6^{p-1}[v,w]_{A_{p}^{+}}\label{eq:ApHuecos}
\end{equation}
a direct application of Lemma \ref{lem:Gap} ends the proof. Hence
it will suffice to settle the latter. We argue as follows. Observe
that 

\begin{align*}
\int_{a}^{b}v\left(\int_{c}^{d}w^{-\frac{1}{\frac{p}{r}-1}}\right)^{\frac{p}{r}-1} & =\int_{a}^{b}v\left(\int_{c}^{d}w^{-\frac{1}{p-1}\frac{p-1}{\frac{p}{r}-1}}\right)^{\frac{1}{\frac{p-1}{\frac{p}{r}-1}}(p-1)}\\
 & =\int_{a}^{b}v\left(\int_{c}^{d}w^{-\frac{1}{p-1}r_{\sigma}}\right)^{\frac{1}{r_{\sigma}}(p-1)}\\
 & \leq(c-b)^{-\frac{r_{\sigma}-1}{r_{\sigma}}(p-1)}2^{\frac{(p-1)}{r_{\sigma}}}\int_{a}^{b}v\left(\int_{b}^{d}w^{-\frac{1}{p-1}}\right)^{(p-1)}\\
 & \leq(c-b)^{-\frac{r_{\sigma}-1}{r_{\sigma}}(p-1)}2^{\frac{(p-1)}{r_{\sigma}}}(d-a)^{p}\frac{\int_{a}^{b}v}{d-a}\left(\frac{\int_{b}^{d}w^{-\frac{1}{p-1}}}{d-a}\right)^{(p-1)}\\
 & \leq(c-b)^{-\frac{r_{\sigma}-1}{r_{\sigma}}(p-1)}2^{p-1}(d-a)^{p}[v,w]_{A_{p}^{+}}\\
 & =\left(\frac{d-a}{3}\right)^{-\frac{r_{\sigma}-1}{r_{\sigma}}(p-1)}2^{p-1}(d-a)^{p}[v,w]_{A_{p}^{+}}\\
 & =(d-a)^{p-\frac{r_{\sigma}-1}{r_{\sigma}}(p-1)}3^{\frac{r_{\sigma}-1}{r_{\sigma}}(p-1)}2^{p-1}[v,w]_{A_{p}^{+}}\\
 & =6^{p-1}(d-a)^{\frac{p}{r}}[v,w]_{A_{p}^{+}}
\end{align*}
and we are done.
\end{proof}
We do not provide a full proof of Theorem \ref{thm:LorentzAp+Joint}
but some hints. To settle that result it suffices to follow the argument
that we provided in the proof of Theorem \ref{thm:LorentzAp+} just
replacing Lemma \ref{lem:RH} by Lemma \ref{lem:RHJoint} and the
role of (\ref{eq:WeakM+}) by the following estimate 
\[
\|M^{+}f\|_{L^{p,\infty}(v)}\lesssim[v,w]_{A_{p}^{+}}^{\frac{1}{p}}\|f\|_{L^{p}(w)}
\]
that was obtained, although not explicitly stated, in \cite{MR}.

\subsection{Proof of Theorem \ref{thm:Conjugation}}

To settle Theorem \ref{thm:Conjugation} we begin settling a lemma
that tells us that if we perturb a weight $A_{p}^{+}$ by $e^{tb}$
where $b$ is a $BMO$ function and $t$ is small enough, then the
weight remains an $A_{p}^{+}$ weight.
\begin{lem}
\label{lem:etbwleqw}Let $1<p<\infty$. If $w\in A_{p}^{+}$ and $b\in BMO$
then there exists $\varepsilon_{p}>0$ such that for $0<|t|\leq\frac{\varepsilon_{p}}{[w]_{A_{p}^{+}}^{\max\left\{ 1,\frac{1}{p-1}\right\} }\|b\|_{BMO}}$
we have that 
\[
[e^{tb}w]_{A_{p}^{+}}\lesssim[w]_{A_{p}^{+}}.
\]
\end{lem}

\begin{proof}[Proof of Lemma \ref{lem:etbwleqw}]
We are going to prove that if
\[
a<b<m<c<d
\]
where
\[
b=a+\frac{d-a}{4}\qquad c=d-\frac{d-a}{4}
\]
and
\[
m=\frac{d+a}{2}
\]
we have that 
\begin{equation}
\frac{1}{d-a}\int_{a}^{b}e^{tb}w\left(\frac{1}{d-a}\int_{c}^{d}(e^{tb}w)^{-\frac{1}{p-1}}\right)^{p-1}\lesssim[w]_{A_{p}^{+}}.\label{eq:ApHuecos-1}
\end{equation}
Observe that provided we can settle such an estimate, in virtue of
Lemma \ref{lem:Gap} with $t=4$ we will be done. Let us settle (\ref{eq:ApHuecos-1})
then. To this end we are going to follow ideas in \cite{HComm,CPP}.
First we observe that
\begin{align*}
 & \frac{1}{d-a}\int_{a}^{b}e^{tb}w\left(\frac{1}{d-a}\int_{c}^{d}(e^{tb}w)^{-\frac{1}{p-1}}\right)^{p-1}\\
 & =\frac{1}{d-a}\int_{a}^{b}e^{tb-tb_{(a,d)}}w\left(\frac{1}{d-a}\int_{c}^{d}(e^{tb-tb_{(a,d)}}w)^{-\frac{1}{p-1}}\right)^{p-1}
\end{align*}
and hence, it suffices to show that
\[
\frac{1}{d-a}\int_{a}^{b}e^{t(b-b_{(a,d)})}w\left(\frac{1}{d-a}\int_{c}^{d}\left(e^{t(b-b_{(a,d)})}w\right)^{-\frac{1}{p-1}}\right)^{p-1}\lesssim[w]_{A_{p}^{+}}.
\]
Let $\tilde{\tau}=\max\{\tau_{p},\tau_{p'}\}$ where $\tau_{p}$ and
$\tau_{p'}$ are chosen as in Lemma \ref{lem:OneSidedRHIAp}, and
let 
\[
\varepsilon=\frac{1}{\tilde{\tau}\max\left\{ [w]_{A_{p}^{+}},\left[w^{-\frac{1}{p-1}}\right]_{A_{p'}^{-}}\right\} }=\frac{1}{\tilde{\tau}[w]_{A_{p}^{+}}^{\max\left\{ 1,\frac{1}{p-1}\right\} }}.
\]
Observe that by the Hölder inequality
\begin{align*}
\int_{a}^{b}e^{t(b-b_{(a,d)})} & w\leq\left(\int_{a}^{b}e^{t(b-b_{(a,d)})(1+\varepsilon)'}\right)^{\frac{1}{(1+\varepsilon)'}}\left(\int_{a}^{b}w^{1+\varepsilon}\right)^{\frac{1}{(1+\varepsilon)}}.
\end{align*}
For the second term by the reverse Hölder inequality in Lemma \ref{lem:OneSidedRHIAp},
\[
\left(\int_{a}^{b}w^{1+\varepsilon}\right)^{\frac{1}{(1+\varepsilon)}}\leq2\frac{1}{|(b,m)|^{\frac{\varepsilon}{1+\varepsilon}}}\int_{a}^{m}w
\]
and hence we have that
\[
\int_{a}^{b}e^{t(b-b_{(a,d)})}w\leq2\left(\int_{a}^{b}e^{t(b-b_{(a,d)})(1+\varepsilon)'}\right)^{\frac{1}{(1+\varepsilon)'}}\frac{1}{|(b,m)|^{\frac{\varepsilon}{1+\varepsilon}}}\int_{a}^{m}w.
\]
Arguing analogously
\begin{align*}
 & \left(\int_{c}^{d}(e^{t(b-b_{(a,d)})}w)^{-\frac{1}{p-1}}\right)^{p-1}\\
 & \leq2\left(\int_{c}^{d}e{}^{-\frac{t(b-b_{(a,d)})(1+\varepsilon)'}{p-1}}\right)^{\frac{p-1}{(1+\varepsilon)'}}\left(\frac{1}{|(m,c)|^{\frac{\varepsilon}{1+\varepsilon}}}\int_{m}^{d}w^{-\frac{1}{p-1}}\right)^{p-1}.
\end{align*}
Combining the estimates above
\begin{align*}
 & \frac{1}{d-a}\int_{a}^{b}e^{t(b-b_{(a,d)})}w\left(\frac{1}{d-a}\int_{c}^{d}(e^{t(b-b_{(a,d)})}w)^{-\frac{1}{p-1}}\right)^{p-1}\\
 & \lesssim\frac{1}{(d-a)^{p}}\left(\frac{1}{|(b,m)|}\int_{a}^{b}e^{t(b-b_{(a,d)})(1+\varepsilon)'}\right)^{\frac{1}{(1+\varepsilon)'}}\\
 & \times\int_{a}^{m}w\left(\frac{1}{|(m,c)|}\int_{c}^{d}e{}^{-\frac{t(b-b_{(a,d)})b(1+\varepsilon)'}{p-1}}\right)^{\frac{p-1}{(1+\varepsilon)'}}\left(\int_{m}^{d}w^{-\frac{1}{p-1}}\right)^{p-1}\\
 & =\frac{1}{(d-a)^{p}}\int_{a}^{m}w\left(\int_{m}^{d}w^{-\frac{1}{p-1}}\right)^{p-1}\\
 & \times\left(\frac{1}{|(b,m)|}\int_{a}^{b}e^{t(b-b_{(a,d)})(1+\varepsilon)'}\right)^{\frac{1}{(1+\varepsilon)'}}\left(\frac{1}{|(m,c)|}\int_{c}^{d}e{}^{-\frac{t(b-b_{(a,d)})(1+\varepsilon)'}{p-1}}\right)^{\frac{p-1}{(1+\varepsilon)'}}\\
 & \leq[w]_{A_{p}^{+}}\left(\frac{1}{|(b,m)|}\int_{a}^{b}e^{t(b-b_{(a,d)})(1+\varepsilon)'}\right)^{\frac{1}{(1+\varepsilon)'}}\left(\frac{1}{|(m,c)|}\int_{c}^{d}e{}^{-\frac{t(b-b_{(a,d)})(1+\varepsilon)'}{p-1}}\right)^{\frac{p-1}{(1+\varepsilon)'}}.
\end{align*}
Finally, choosing $|t|\leq\frac{\varepsilon_{p}}{[w]_{A_{p}^{+}}^{\max\left\{ 1,\frac{1}{p-1}\right\} }\|b\|_{BMO}}$
for some suitable $\varepsilon_{p}>0$, it is possible to show arguing
as in \cite{HComm} that
\begin{equation}
\left(\frac{1}{|(a,d)|}\int_{a}^{d}e^{t(b-b_{(a,d)})(1+\varepsilon)'}\right)^{\frac{1}{(1+\varepsilon)'}}\left(\frac{1}{|(a,d)|}\int_{a}^{d}e{}^{-\frac{t(b-b_{(a,d)})(1+\varepsilon)'}{p-1}}\right)^{\frac{p-1}{(1+\varepsilon)'}}\lesssim1.\label{eq:exp(b-bI)Ap}
\end{equation}
We include here an argument showing this estimate holds for reader's
convenience. A version of John-Nirenberg inequality (see \cite[p. 31]{J})
states that there exist dimensional constants $\lambda>0$ and $c>1$
such that for every $b\in BMO(\mathbb{R}^{n})$ 
\[
\frac{1}{|I|}\int_{I}e^{\frac{\lambda|b-b_{I}|}{\|b\|_{BMO}}}\leq c.
\]
Bearing that estimate in mind we observe that choosing $|t|\leq\frac{\lambda}{(1+\varepsilon)^{'}\|b\|_{BMO}}$
\begin{align}
\left(\frac{1}{|(a,d)|}\int_{a}^{d}e^{t(b-b_{(a,d)})(1+\varepsilon)'}\right)^{\frac{1}{(1+\varepsilon)'}} & \leq\left(\frac{1}{|(a,d)|}\int_{a}^{d}e^{|t||b-b_{(a,d)}|(1+\varepsilon)'}\right)^{\frac{1}{(1+\varepsilon)'}}\label{eq:expAp1}\\
 & \leq\left(\frac{1}{|(a,d)|}\int_{a}^{d}e^{\frac{|b-b_{(a,d)}|\lambda}{\|b\|_{BMO}}}\right)^{\frac{1}{(1+\varepsilon)'}}\leq c^{\frac{1}{(1+\varepsilon)'}}\leq c.
\end{align}
On the other hand, choosing $|t|\leq\frac{\lambda(p-1)}{(1+\varepsilon)^{'}\|b\|_{BMO}}$
we have that
\begin{equation}
\begin{split}\left(\frac{1}{|(a,d)|}\int_{a}^{d}e{}^{-\frac{t(b-b_{(a,d)})(1+\varepsilon)'}{p-1}}\right)^{\frac{p-1}{(1+\varepsilon)'}} & \leq\left(\frac{1}{|(a,d)|}\int_{a}^{d}e{}^{\frac{|t||b-b_{(a,d)}|(1+\varepsilon)'}{p-1}}\right)^{\frac{p-1}{(1+\varepsilon)'}}\\
 & \leq\left(\frac{1}{|(a,d)|}\int_{a}^{d}e^{\frac{|b-b_{(a,d)}|\lambda}{\|b\|_{BMO}}}\right)^{\frac{p-1}{(1+\varepsilon)'}}\leq c^{\frac{p-1}{(1+\varepsilon)'}}\leq c^{p}.
\end{split}
\label{eq:expAp2}
\end{equation}
Consequently, since $(1+\varepsilon)'\leq2\tilde{\tau}[w]_{A_{p}^{+}}^{\max\left\{ 1,\frac{1}{p-1}\right\} }$
we have that taking 
\[
|t|=\frac{\min\left\{ p-1,1\right\} }{2\tilde{\tau}}\lambda\frac{1}{[w]_{A_{p}^{+}}^{\max\left\{ 1,\frac{1}{p-1}\right\} }\|b\|_{BMO}}=\frac{\varepsilon_{p}}{[w]_{A_{p}^{+}}^{\max\left\{ 1,\frac{1}{p-1}\right\} }\|b\|_{BMO}}
\]
clearly $|t|\leq\min\left\{ \frac{\lambda(p-1)}{(1+\varepsilon)^{'}\|b\|_{BMO}},\frac{\lambda}{(1+\varepsilon)^{'}\|b\|_{BMO}}\right\} $,
and hence both (\ref{eq:expAp1}) and (\ref{eq:expAp2}) hold which
in turn implies that (\ref{eq:exp(b-bI)Ap}) holds as well, as we
wanted to show.
\end{proof}
Now we are in the position to settle Theorem \ref{thm:Conjugation}.
\begin{proof}[Proof of Theorem \ref{thm:Conjugation}]
Let
\[
T_{z}f(x)=e^{b(x)z}T(e^{-bz}f)(x)\qquad z\in\mathbb{C}
\]
Observe that 
\[
T_{b}^{k}f(x)=\left.\frac{d^{k}}{dz^{k}}T_{z}f\right|_{z=0}(x)=\frac{k!}{2\pi i}\int_{|z|=\delta}\frac{T_{z}f(x)}{z^{k+1}}dz.
\]
Now integrating 
\begin{align*}
\|T_{b}^{k}f\|_{L^{q}(w^{q})} & \leq\frac{k!}{2\pi}\left\Vert \int_{|z|=\delta}\frac{T_{z}f}{z^{k+1}}dz\right\Vert _{L^{q}(w^{q})}\\
 & \leq\frac{k!}{2\pi}\int_{|z|=\delta}\|T_{z}f\|_{L^{q}(w^{q})}\frac{|dz|}{|z^{k+1}|}\\
 & \leq\frac{k!}{\delta^{k}}\sup_{|z|=\delta}\|T_{z}f\|_{L^{q}(w^{q})}\\
 & =\frac{k!}{\delta^{k}}\sup_{|z|=\delta}\|e^{b(x)Re(z)}T\left(e^{-bz}f\right)\|_{L^{p}(w^{q})}\\
 & =\frac{k!}{\delta^{k}}\sup_{|z|=\delta}\|T\left(e^{-bz}f\right)\|_{L^{q}(e^{bRe(z)q}w^{q})}.
\end{align*}
Observe that by the Lemma above, choosing $\delta=\frac{\varepsilon_{p_{0}}}{[w^{q}]_{A_{p_{0}}^{+}}^{\max\left\{ 1,\frac{1}{p_{0}-1}\right\} }\|b\|_{BMO}}$
for a suitable $\varepsilon_{p_{0}}$, we have that 
\begin{align*}
\|T\left(e^{-bz}f\right)\|_{L^{q}(e^{bRe(z)q}w^{q})} & \lesssim\varphi\left(\left[e^{bRe(z)q}w^{q}\right]_{A_{p_{0}}^{+}}\right)\|e^{-bRe(z)}f\|_{L^{p}(e^{bRe(z)p}w^{p})}\\
 & \lesssim\varphi\left(\kappa[w^{q}]_{A_{p_{0}}^{+}}\right)\|f\|_{L^{p}(w^{p})}.
\end{align*}
Hence
\begin{align*}
\|T_{b}^{k}f\|_{L^{q}(w^{q})} & \lesssim\frac{k!}{\delta^{k}}\sup_{|z|=\delta}\|T\left(e^{-bz}f\right)\|_{L^{p}(e^{bRe(z)q}w^{q})}\\
 & \lesssim\varphi\left(\kappa[w^{q}]_{A_{p_{0}}^{+}}\right)\frac{1}{\delta^{k}}\|f\|_{L^{p}(w^{p})}\\
 & \lesssim\varphi\left(\kappa[w^{q}]_{A_{p_{0}}^{+}}\right)[w^{q}]_{A_{p_{0}}^{+}}^{k\max\left\{ 1,\frac{1}{p_{0}-1}\right\} }\|b\|_{BMO}^{k}\|f\|_{L^{p}(w^{p})}
\end{align*}
and we are done.
\end{proof}

\section{Corollaries of Theorem \ref{thm:Conjugation}\label{sec:CorConjugation}}

In this section we derive some corollaries of Theorem \ref{thm:Conjugation}.
We begin with the following one-sided counterpart of the bound obtained
in \cite{BMMST,CUM}
\begin{cor}
Let $0<\alpha<1$, $1<p<\frac{1}{\alpha}$ , $\frac{1}{q}=\frac{1}{p}-\alpha$
and $w\in A_{p,q}^{+}$. Let $b\in BMO$. Then
\[
\left\Vert \left(I_{\alpha}^{+}\right)_{b}^{k}\right\Vert _{L^{q}(w^{q})}\lesssim\|b\|_{BMO}^{k}[w]_{A_{p,q}^{+}}^{((k+1)-\alpha)\max\left\{ 1,\frac{p'}{q}\right\} }\|f\|_{L^{p}(w^{p})}.
\]
\end{cor}

\begin{proof}
We follow ideas in \cite{CUM}. We begin settling the case 
\[
\frac{2}{p}=1+\alpha.
\]
Note that in this case we have that $q=p'$ and hence 
\[
[w]_{A_{p,p'}^{+}}=[w^{p'}]_{A_{2}^{+}}=[w^{-p'}]_{A_{2}^{-}}.
\]
We also know that
\[
\|I_{\alpha}^{+}f\|_{L^{p'}(w^{p'})}\lesssim[w]_{A_{p,p'}^{+}}^{1-\alpha}\|f\|_{L^{p}(w^{p}).}
\]
or equivalently
\[
\|I_{\alpha}^{+}f\|_{L^{p'}(w^{p'})}\lesssim[w^{p'}]_{A_{2}^{+}}^{1-\alpha}\|f\|_{L^{p}(w^{p}).}
\]
Observe that by Theorem \ref{thm:Conjugation}, choosing $p=p$, $q=p'$
and $p_{0}=2$, we have that
\[
\left\Vert \left(I_{\alpha}^{+}\right)_{b}^{k}f\right\Vert _{L^{p'}(w^{p'})}\lesssim\|b\|_{BMO}^{k}[w^{p'}]_{A_{2}^{+}}^{(1-\alpha)+k}\|f\|_{L^{p}(w^{p})}
\]
which in turn is equivalent to
\[
\left\Vert \left(I_{\alpha}^{+}\right)_{b}^{k}f\right\Vert _{L^{p'}(w^{p'})}\lesssim\|b\|_{BMO}^{k}[w]_{A_{p,p'}^{+}}^{((k+1)-\alpha)}\|f\|_{L^{p}(w^{p})}.
\]
Now extrapolation in Theorem \ref{thm:SharpExt} shows that 
\[
\left\Vert \left(I_{\alpha}^{+}\right)_{b}^{k}f\right\Vert _{L^{q}(w^{q})}\lesssim\|b\|_{BMO}^{k}[w]_{A_{p,q}^{+}}^{((k+1)-\alpha)\max\left\{ 1,\frac{p'}{q}\right\} }\|f\|_{L^{p}(w^{p})}
\]
and hence we are done.
\end{proof}
Our next result provides a quantitative estimate for the martingale
transform $G$ studied in \cite{CHL}.
\begin{cor}
Let $1<p<\infty$ and $w\in A_{p}^{+}$. Let $b\in BMO$. Then
\[
\|G_{b}^{k}f\|_{L^{q}(w^{q})}\lesssim\|b\|_{BMO}^{k}[w]_{A_{p}^{+}}^{(k+1)\max\left\{ 1,\frac{1}{p-1}\right\} }\|f\|_{L^{p}(w)}.
\]
\end{cor}

\begin{proof}
First we observe that by Theorem \ref{thm:SharpExt} combined with
(\ref{eq:A2Martingale}) we have that for every $p>1$ and every $w\in A_{p}^{+}$,
\[
\|Gf\|_{L^{p}(w)}\lesssim[w]_{A_{p}^{+}}^{\max\left\{ 1,\frac{1}{p-1}\right\} }\|f\|_{L^{p}(w)}.
\]
Observe that this says that for every $w^{p}\in A_{p}^{+}$
\[
\|Gf\|_{L^{p}(w^{p})}\lesssim[w^{p}]_{A_{p}^{+}}^{\max\left\{ 1,\frac{1}{p-1}\right\} }\|f\|_{L^{p}(w^{p})}.
\]
Now we have that using Theorem \ref{thm:Conjugation} with $p=p$,
$q=p$, $p_{0}=p$, we have that for every $w^{p}\in A_{p}^{+}$,
\[
\|G_{b}^{k}f\|_{L^{p}(w^{p})}\lesssim\|b\|_{BMO}^{k}[w^{p}]_{A_{p}^{+}}^{(k+1)\max\left\{ 1,\frac{1}{p-1}\right\} }\|f\|_{L^{p}(w^{p})}
\]
and hence we are done.

Using the Theorem \ref{thm:Conjugation} we can provide as well estimates
for commutators of $b\in BMO$ and one-sided $L^{r'}$-Hörmander operators
or Calderón-Zygmund operators even though the precise dependence on
the $A_{p/r}^{+}$ constant remains an open problem. Since in particular
Calderón-Zygmund operators are $L^{\infty}$-Hörmander operators we
may provide an unified statement.
\end{proof}
\begin{cor}
Let $1\leq r<p<\infty$. Let $b\in BMO$ and $T^{+}$ an $L^{r'}$-Hörmander
operator. If for every $w\in A_{p/r}^{+}$ we have that
\[
\|T^{+}f\|_{L^{p}(w)}\lesssim\varphi_{r}([w]_{A_{p/r}^{+}})\|f\|_{L^{p}(w)},
\]
then
\[
\|(T^{+})_{b}^{k}f\|_{L^{p}(w)}\lesssim\|b\|_{BMO}^{k}\varphi_{r}([w]_{A_{p/r}^{+}})[w]_{A_{p/r}^{+}}^{k\max\left\{ 1,\frac{r}{p-r}\right\} }\|f\|_{L^{p}(w)}.
\]
\end{cor}

\begin{proof}
Using Theorem \ref{thm:Conjugation} with $p\text{=}p$, $q=p$, $p_{0}=p/r$,
we have that for every $w^{p}\in A_{p/r}^{+}$, 
\[
\|(T^{+})_{b}^{k}f\|_{L^{p}(w^{p})}\lesssim\|b\|_{BMO}^{k}\varphi_{r}([w]_{A_{p/r}^{+}})[w^{p}]_{A_{p/r}^{+}}^{k\max\left\{ 1,\frac{r}{p-r}\right\} }\|f\|_{L^{p}(w^{p})}
\]
and hence we are done.
\end{proof}
In view of the known results from the classical setting, namely, that
if $T$ is an $L^{r'}$-Hörmander operator, we have that 
\[
\|T^{+}f\|_{L^{p}(w)}\lesssim[w]_{A_{p/r}^{+}}^{\max\left\{ 1,\frac{1}{p-r}\right\} }\|f\|_{L^{p}(w)}
\]
It is natural to conjecture that 
\[
\|(T^{+})_{b}^{k}f\|_{L^{p}(w)}\lesssim\|b\|_{BMO}^{k}[w]_{A_{p/r}^{+}}^{k\max\left\{ 1,\frac{r}{p-r}\right\} +\max\left\{ 1,\frac{1}{p-r}\right\} }\|f\|_{L^{p}(w)}.
\]
Note that if $r=1$ the dependence would be the same as in \cite{CPP,HComm}.

\section*{Acknowledgment }

The first and the second authors were partially supported by Ministerio
de Economía y Competitividad, Spain, grant PGC2018-096166-B-I00 and
by Junta de Andalucía grant FQM-354. \\
The third author was partially supported by FONCyT grants PICT 2018-02501
and PICT 2019-00018. \\
All the authors were partially supported by Junta de Andalucía UMA18-FEDERJA002.

\bibliographystyle{abbrv}
\bibliography{refs}

\end{document}